\documentclass[12pt]{amsart}

\usepackage[T1]{fontenc}
\usepackage{amsmath,amssymb,amsthm,mathtools}
\usepackage[margin=1in]{geometry}
\usepackage{enumitem}
\usepackage{microtype}
\usepackage[hidelinks]{hyperref}

\newcommand{\Bc}{\mathcal{B}}

\theoremstyle{plain}
\newtheorem{theorem}{Theorem}[section]
\newtheorem{proposition}[theorem]{Proposition}
\newtheorem{lemma}[theorem]{Lemma}

\theoremstyle{definition}

\theoremstyle{remark}
\newtheorem{remark}[theorem]{Remark}

\title[Log-concavity and unimodality for Hilbert schemes of points]
      {Log-concavity and unimodality of Hodge numbers\\
       of Hilbert schemes of points over a Surface}
\author{Anubhab Pahari}
\address{Department of Mathematics, IIT Madras, Chennai, India}
\email{anubhabpahari@gmail.com, ma22d012@smail.iitm.ac.in}

\begin{document}

\begin{abstract}
Let \(S\) be a smooth projective complex surface with irregularity
\(q=h^{1,0}(S)\) and geometric genus \(g=h^{2,0}(S)\), and let
\(S^{[n]}\) denote its Hilbert scheme of \(n\) points.  We prove that,
for every \(n\ge0\), the sequence
\[
\left(h^{p,0}\bigl(S^{[n]}\bigr)\right)_{p=0}^{2n}
\]
is log-concave if and only if
\(g\le\binom{q+1}{2}\).  Moreover, log-concavity of all these sequences
is already equivalent to log-concavity of the sequence for \(n=2\).
We also prove that these sequences are unimodal for every \(n\ge0\) if
and only if \(q\ge1\) or \(g=0\), and that this condition is already
detected by the sequence for \(n=1\).
\end{abstract}
\maketitle

\section{Introduction}\label{sec:introduction}

A finite or infinite sequence \((u_j)\) of nonnegative real numbers is
\emph{log-concave} if
\[
 u_j^2\ge u_{j-1}u_{j+1}
\]
whenever the three terms occur.  A finite sequence
\((v_0,\ldots,v_D)\) is \emph{unimodal} if there is an index \(\mu\)
such that
\[
 v_0\le\cdots\le v_\mu\ge\cdots\ge v_D.
\]
Every finite positive log-concave sequence is unimodal, since its
consecutive ratios are nonincreasing.  We refer to
\cite{Stanley,Brenti} for
background on these properties.

Let \(S\) be a smooth projective surface over \(\mathbb C\), and let
\(S^{[n]}\) denote the Hilbert scheme of points of \(S\) of length
\(n\).  Fogarty proved that \(S^{[n]}\) is smooth of complex dimension
\(2n\) \cite[Theorem~2.4]{Fogarty}.  Put
\[
 q:=h^{1,0}(S),
 \qquad
 g:=h^{2,0}(S),
\]
the irregularity and geometric genus of \(S\), respectively.  A
specialization of G\"ottsche's formula
\cite[Proposition~3.3(a)]{Gottsche}(see \cite{GottscheSoergel}, for general case) gives
\[
 \sum_{n\ge0}\sum_{p\ge0}
 h^{p,0}\bigl(S^{[n]}\bigr)x^pt^n
 =
 \frac{(1+xt)^q}{(1-t)(1-x^2t)^g}.
\]
Thus the Hodge rows
\[
 \left(h^{p,0}\bigl(S^{[n]}\bigr)\right)_{p=0}^{2n}
\]
depend only on \(q\) and \(g\).  Our aim is to determine exactly when
all these rows are log-concave or unimodal.

Sequences arising from Hilbert schemes have previously been studied
primarily by varying \(n\).  For example, suppose that the \emph{topological Euler characteristic}
\(\chi_{\mathrm{top}}(S)=k\ge0\), and let \(p_k(n)\) denote the number
of \(k\)-coloured partitions of \(n\).  G\"ottsche's formula for $\chi_{\mathrm{top}}(S^{[n]})$ \cite[Theorem 0.1 (2)]{Gottsche} gives
\[
 \chi_{\mathrm{top}}\bigl(S^{[n]}\bigr)
 =
 p_k(n).
\]
For \(k\ge3\), the log-concavity of \(\bigl(p_k(n)\bigr)_{n\ge0}\) was
proved in \cite{Bringmann1} and subsequently strengthened in
\cite{Bringmann2} to
\[
 p_k(n)^2>p_k(n-1)p_k(n+1)
 \qquad(n\ge1),
\]
with the single exception \((k,n)=(3,1)\).  Related asymptotic and
distributional questions for Betti and Hodge numbers of Hilbert
schemes were studied in
\cite{BringmannManschot,ManschotZapataRolon,GillmanEtAl}.  In contrast,
we fix \(n\) and vary the holomorphic degree \(p\).

Our first result gives an exact numerical criterion for log-concavity.

\begin{theorem}\label{thm:main-hodge}
The following conditions are equivalent:
\begin{enumerate}[label=\textup{(\roman*)}]
\item \(g\le\binom{q+1}{2}\);
\item for every \(n\ge0\), the sequence
      \[
       \left(h^{p,0}\bigl(S^{[n]}\bigr)\right)_{p=0}^{2n}
      \]
      is log-concave;
\item
      \[
       \bigl(h^{1,0}(S^{[2]})\bigr)^2
       \ge
       h^{0,0}(S^{[2]})h^{2,0}(S^{[2]}).
      \]
\end{enumerate}
\end{theorem}
The threshold is sharp.  If \(S\) is a K3 surface, then \((q,g)=(0,1)\), and the row for \(n=2\) is $(1,0,1,0,1),$ which is not log-concave. 

The proof of Theorem \ref{thm:main-hodge} mainly rests on a reformulation of $h^{p,0}(S^{[n]})$ given by the sum $T_m(W)$ (see Lemma \ref{lem:two-coefficient-forms}): For nonnegative integers \(q\) and \(g\),
define \(c_j\) and \(T_m(W)\) by
\[
 \sum_{j\ge0}c_jz^j=(1-z)^{-g},
 \qquad
 T_m(W):=\sum_{r\ge0}\binom q{m-2r}c_{W+r},
\]
where \(W,m\ge0\) and
\(\binom qk=0\) unless \(0\le k\le q\).
Theorem~\ref{thm:universal-tail} establishes the equivalence
\[
 g\le\binom{q+1}{2}
 \quad\Longleftrightarrow\quad
 T_m(W)^2\ge T_{m-1}(W)T_{m+1}(W)
 \quad\text{for every }W\ge0\text{ and }m\ge1.
\]

The condition in Theorem~\ref{thm:main-hodge} holds for many standard
classes of surfaces \cite{BHPV}.  It holds for every surface with
\(g=0\) and, in particular, for rational, Enriques, ruled, and
bielliptic surfaces.  It also holds for abelian surfaces.  If
\(S=C_r\times C_s\), where \(C_r\) and \(C_s\) are curves of genus
\(r\) and \(s\), then \(q=r+s\), \(g=rs\), and
\[
 \binom{q+1}{2}-g
 =
 \binom{r+1}{2}+\binom{s+1}{2}
 \ge0.
\]
Note that, \(q\) and \(g\) are birational invariants of smooth projective surfaces, so the condition is unchanged by blowing up points.

Although log-concavity can fail above this threshold, unimodality is
governed by a different and simpler condition.

\begin{theorem}\label{thm:unimodality-intro}
The following conditions are equivalent:
\begin{enumerate}[label=\textup{(\roman*)}]
\item \(q\ge1\) or \(g=0\);
\item for every \(n\ge0\), the sequence
      \[
       \left(h^{p,0}\bigl(S^{[n]}\bigr)\right)_{p=0}^{2n}
      \]
      is unimodal;
\item the sequence
      \[
       \left(h^{p,0}\bigl(S^{[1]}\bigr)\right)_{p=0}^{2}
      \]
      is unimodal.
\end{enumerate}
\end{theorem}

Here necessity is already visible for \(n=1\), since
\(S^{[1]}\cong S\) and its Hodge row is \((1,q,g)\).  If \(q=0\) and
\(g>0\), this row is not unimodal (see remark \ref{rem:q-zero-unimodality}).

We conclude with an outline of the proofs.  In
Section~\ref{sec:coefficients}, we extract an explicit coefficient
formula from G\"ottsche's generating series and reduce log-concavity to some universal inequalities.  Section~\ref{sec:tail-proof} proves
these inequalities and Theorem~\ref{thm:main-hodge}.
Section~\ref{sec:unimodality} studies the first differences of the
Hodge rows and proves Theorem~\ref{thm:unimodality-intro}.

\section{Coefficient formulas and the universal-tail reduction}
\label{sec:coefficients}
The $(p,r)^{\text{th}}$ \emph{Hodge number} of a smooth projective complex variety $X$ of dimension $d$ is defined as $$h^{p,r}(X)= \dim_{\mathbb{C}}H^r(X, \Omega^p_X),\;\; 0\le p,r \le d;$$ where $\Omega^p_X$ is the sheaf of holomorphic $p$-forms.

Set for $r=0$
\[
 a_p^{(n)}:=h^{p,0}\bigl(S^{[n]}\bigr)
 \qquad(n,p\ge0);
\]
thus \(a_p^{(n)}=0\) when \(p>2n\).
G\"ottsche's formula \cite[Proposition~3.3(a)]{Gottsche} gives
\begin{equation}\label{eq:master-series}
 \sum_{n\ge0}\sum_{p\ge0}a_p^{(n)}x^pt^n
 =\frac{(1+xt)^q}{(1-t)(1-x^2t)^g}.
\end{equation}
Define the coefficients \(c_j\) by
\begin{equation}\label{eq:c-generating}
 \sum_{j\ge0}c_jz^j=(1-z)^{-g},
 \qquad c_j:=0\quad(j<0).
\end{equation}
Thus, for \(j\ge0\),
\[
 c_j=\binom{g+j-1}{j}\quad(g\ge1),
\]
whereas \(c_0=1\) and \(c_j=0\) for \(j>0\) when \(g=0\).
Throughout, we use the convention
\[
 \binom{A}{k}=0\qquad\text{unless }0\le k\le A
\]
for every nonnegative integer \(A\) and every integer \(k\).

\subsection{The coefficient formula}

\begin{lemma}\label{lem:coefficient-formula}
For \(n\ge0\) and \(0\le p\le2n\),
\begin{equation}\label{eq:coefficient-formula}
 a_p^{(n)}=
 \sum_{\substack{0\le i\le\min\{q,p,2n-p\}\\
                   i\equiv p\pmod2}}
 \binom qi\,c_{(p-i)/2}.
\end{equation}
\end{lemma}

\begin{proof}
Choose \(x^it^i\) from \((1+xt)^q\) and \(x^{2j}t^j\) from
\((1-x^2t)^{-g}\).  These choices contribute to the coefficient of
\(x^pt^n\) precisely when
\[
 i+2j=p,
 \qquad
 i+j\le n,
\]
and their contribution is \(\binom qi c_j\).  Substituting
\(j=(p-i)/2\) gives the parity condition on \(i\), while \(i+j\le n\)
is equivalent to \(i\le2n-p\).  This proves
\eqref{eq:coefficient-formula}.
\end{proof}

Formula~\eqref{eq:coefficient-formula} also shows that if \(q,g\ge1\), then
\(a_p^{(n)}>0\) for every \(0\le p\le2n\): in
\eqref{eq:coefficient-formula}, the term with \(i=0\) when \(p\) is even,
or with \(i=1\) when \(p\) is odd, is admissible and positive.

\begin{remark}[The case \(g=0\)]\label{rem:g-zero}
If \(g=0\), then \eqref{eq:coefficient-formula} reduces to
\[
 a_p^{(n)}=
 \begin{cases}
  \binom qp,&0\le p\le n,\\
  0,&n<p\le2n.
 \end{cases}
\]
Thus the row is an initial segment of a binomial row followed by zeros,
and hence is log-concave and unimodal.
\end{remark}

\subsection{Reduction to a universal inequality}

When \(g\ge1\), equation~\eqref{eq:c-generating} gives
\begin{equation}\label{eq:c-ratio}
 \frac{c_{j+1}}{c_j}
 =\frac{g+j}{j+1}
 =1+\frac{g-1}{j+1}
 \qquad(j\ge0).
\end{equation}
Consequently, \((c_j)_{j\ge0}\) is positive and nondecreasing, and its
consecutive ratios are nonincreasing; in particular, it is log-concave.

For arbitrary \(g\ge0\) and \(W,m\ge0\), define the shifted sums
\begin{equation}\label{eq:T-definition}
 T_m(W):=\sum_{r\ge0}\binom q{m-2r}c_{W+r}.
\end{equation}
Separating the term \(r=0\) and shifting the remaining index gives the
useful identity
\begin{equation}\label{eq:T-shift}
 T_{m+1}(W)
 =\binom q{m+1}c_W+T_{m-1}(W+1)
 \qquad(m\ge1,\ W\ge0).
\end{equation}

The following formula identifies the two halves of each Hodge row with
these shifted sums.

\begin{lemma}\label{lem:two-coefficient-forms}
For \(n\ge0\) and \(0\le p\le2n\),
\begin{equation}\label{eq:hodge-to-T}
 a_p^{(n)}=
 \begin{cases}
  T_p(0),&0\le p\le n,\\
  T_{2n-p}(p-n),&n\le p\le2n.
 \end{cases}
\end{equation}
The two expressions agree when \(p=n\).
\end{lemma}

\begin{proof}
If \(p\le n\), then the restriction \(i\le2n-p\) in
\eqref{eq:coefficient-formula} is redundant; setting \(i=p-2r\) gives
\(a_p^{(n)}=T_p(0)\).  If \(p\ge n\), set \(m=2n-p\) and then
\(i=m-2r\) in \eqref{eq:coefficient-formula}.  Since
\[
 \frac{p-i}{2}=p-n+r,
\]
we obtain \(a_p^{(n)}=T_{2n-p}(p-n)\).
\end{proof}

The first three shifted sums are
\[
 T_0(0)=1,
 \qquad
 T_1(0)=q,
 \qquad
 T_2(0)=\binom q2+g.
\]
Hence the inequality at \(W=0\) and \(m=1\) is precisely
\[
 T_1(0)^2\ge T_0(0)T_2(0)
 \quad\Longleftrightarrow\quad
 g\le\binom{q+1}{2}.
\]
The content of the following theorem is that this first necessary
inequality already implies all the shifted inequalities.

\begin{theorem}\label{thm:universal-tail}
Let \(q,g\ge0\) be integers.  The following conditions are equivalent:
\begin{enumerate}[label=\textup{(\roman*)}]
\item \(g\le\binom{q+1}{2}\);
\item for every \(W\ge0\) and \(m\ge1\),
      \begin{equation}\label{eq:universal-tail}
       T_m(W)^2\ge T_{m-1}(W)T_{m+1}(W).
      \end{equation}
\end{enumerate}
\end{theorem}

Theorem~\ref{thm:universal-tail} will be proved in
Section~\ref{sec:tail-proof}.  We next explain why its inequalities imply
log-concavity of the Hodge sequence.  The following shift comparison handles
the right half of a row.

\begin{proposition}\label{prop:freezing}
Assume \(g\ge1\).  For every \(W\ge1\) and \(m\ge1\),
\[
 T_{m-1}(W+1)T_{m+1}(W-1)
 \le T_{m-1}(W)T_{m+1}(W).
\]
\end{proposition}

\begin{proof}
Applying \eqref{eq:T-shift} at \(W\) and \(W-1\), respectively, gives
\[
 \begin{aligned}
 T_{m+1}(W)
   &=T_{m-1}(W+1)+\binom q{m+1}c_W,\\
 T_{m+1}(W-1)
   &=T_{m-1}(W)+\binom q{m+1}c_{W-1}.
 \end{aligned}
\]
Therefore,
\begin{align*}
&T_{m-1}(W)T_{m+1}(W)
 -T_{m-1}(W+1)T_{m+1}(W-1)\\
&\quad=\binom q{m+1}
 \bigl(c_WT_{m-1}(W)-c_{W-1}T_{m-1}(W+1)\bigr)\\
&\quad=\binom q{m+1}
 \sum_{r\ge0}\binom q{m-1-2r}
 \bigl(c_Wc_{W+r}-c_{W-1}c_{W+r+1}\bigr).
\end{align*}
By \eqref{eq:c-ratio}, the consecutive ratios of \(c_j\) are
nonincreasing.  Since \(W-1\le W+r\),
\[
 \frac{c_W}{c_{W-1}}
 \ge \frac{c_{W+r+1}}{c_{W+r}},
\]
so every summand in the last expression is nonnegative.  This proves the
claim.
\end{proof}

\begin{proposition}\label{lem:tail-to-hodge}
Suppose that \eqref{eq:universal-tail} holds for every \(W\ge0\) and
\(m\ge1\).  Then, for every \(n\ge0\), the sequence
\[
 \bigl(a_p^{(n)}\bigr)_{0\le p\le2n}
\]
is log-concave.
\end{proposition}

\begin{proof}
The case \(g=0\) follows from Remark~\ref{rem:g-zero}; hence assume
\(g\ge1\).  There is nothing to prove when \(n=0\).  Fix \(n\ge1\) and
an interior index \(1\le p\le2n-1\).

If \(1\le p\le n-1\), then \eqref{eq:hodge-to-T} gives
\[
 \bigl(a_{p-1}^{(n)},a_p^{(n)},a_{p+1}^{(n)}\bigr)
 =\bigl(T_{p-1}(0),T_p(0),T_{p+1}(0)\bigr),
\]
and the desired inequality is \eqref{eq:universal-tail} with
\((m,W)=(p,0)\).

If \(p=n\), then
\[
 \bigl(a_{n-1}^{(n)},a_n^{(n)},a_{n+1}^{(n)}\bigr)
 =\bigl(T_{n-1}(0),T_n(0),T_{n-1}(1)\bigr).
\]
Moreover, \eqref{eq:T-shift} gives
\[
 T_{n+1}(0)
 =T_{n-1}(1)+\binom q{n+1}c_0
 \ge T_{n-1}(1).
\]
Consequently,
\[
 \bigl(a_n^{(n)}\bigr)^2
 =T_n(0)^2
 \ge T_{n-1}(0)T_{n+1}(0)
 \ge a_{n-1}^{(n)}a_{n+1}^{(n)}.
\]

Finally, suppose that \(n+1\le p\le2n-1\), and put
\[
 m:=2n-p,
 \qquad
 W:=p-n.
\]
Then \(1\le m\le n-1\), \(W\ge1\), and
\[
 \bigl(a_{p-1}^{(n)},a_p^{(n)},a_{p+1}^{(n)}\bigr)
 =\bigl(T_{m+1}(W-1),T_m(W),T_{m-1}(W+1)\bigr).
\]
Using \eqref{eq:universal-tail} and then
Proposition~\ref{prop:freezing}, we obtain
\begin{align*}
 \bigl(a_p^{(n)}\bigr)^2
 &=T_m(W)^2\\
 &\ge T_{m-1}(W)T_{m+1}(W)\\
 &\ge T_{m-1}(W+1)T_{m+1}(W-1)\\
 &=a_{p-1}^{(n)}a_{p+1}^{(n)}.
\end{align*}
These cases exhaust all interior indices, so the row is log-concave.
\end{proof}

\section{Proof of Theorem~\ref{thm:universal-tail} and
Theorem~\ref{thm:main-hodge}}
\label{sec:tail-proof}

Except where stated otherwise, we assume
\begin{equation}\label{eq:tail-standing-assumptions}
 q\ge2,\qquad g\ge1.
\end{equation}

\subsection{An equivalent criterion}

Put
\[
 N:=q+2,\qquad
 L:=N-1=q+1,\qquad
 \Bc:=\binom{q+1}{2}=\binom L2.
\]
For fixed \(W\ge0\), write
\[
 \alpha_r:=c_{W+r}\qquad(r\ge0).
\]
Equation~\eqref{eq:c-ratio} gives
\begin{equation}\label{eq:alpha-ratio}
 \frac{\alpha_{r+1}}{\alpha_r}
 =\frac{g+W+r}{W+r+1}
 =1+\frac{g-1}{W+r+1}.
\end{equation}
Thus \((\alpha_r)_{r\ge0}\) is positive and nondecreasing, and its
consecutive quotients are nonincreasing.

Fix \(m\ge1\).  For \(0\le u\le N\), define
\begin{equation}\label{eq:w-definition}
 w_u:=
 \begin{cases}
  \displaystyle\binom Nu\alpha_{(m+1-u)/2},
   &u\equiv m+1\pmod2\ \text{and }u\le m+1,\\[5pt]
  0,&\text{otherwise},
 \end{cases}
\end{equation}
and set
\begin{equation}\label{eq:moment-sums}
 Z:=\sum_{u=0}^{N}w_u,\qquad
 V:=\sum_{u=0}^{N}u w_u,\qquad
 \Sigma:=\sum_{u=0}^{N}u(u-1)w_u.
\end{equation}

Using \(q=N-2\), we have
\[
 \binom q{u-1}
 =\frac{u(N-u)}{N(N-1)}\binom Nu,\quad
 \binom q{u-2}
 =\frac{u(u-1)}{N(N-1)}\binom Nu,
\]
\[
 \binom qu
 =\frac{(N-u)(N-u-1)}{N(N-1)}\binom Nu.
\]
Substituting \(u=m+1-2r\) in the three sums gives
\begin{equation}\label{eq:T-three-forms}
\begin{aligned}
 N(N-1)T_{m-1}(W)
 &=\Sigma,\\
 N(N-1)T_m(W)
 &=(N-1)V-\Sigma,\\
 N(N-1)T_{m+1}(W)
 &=N(N-1)Z-2(N-1)V+\Sigma.
\end{aligned}
\end{equation}

\begin{proposition}\label{prop:master-identity}
Let \(q,g,W\) be nonnegative integers, and let \(m\ge1\) be an integer.
Then
\begin{equation}\label{eq:master-identity}
\begin{aligned}
 &N^2(N-1)
 \bigl(T_m(W)^2-T_{m-1}(W)T_{m+1}(W)\bigr)\\
 &\hspace{35mm}=(N-1)V^2-NZ\Sigma.
\end{aligned}
\end{equation}
For these fixed \(W\) and \(m\), the inequality
\eqref{eq:universal-tail} is equivalent to
\begin{equation}\label{eq:moment-goal}
 (N-1)V^2\ge NZ\Sigma.
\end{equation}
\end{proposition}

\begin{proof}
Substituting \eqref{eq:T-three-forms}, we obtain
\begin{align*}
 &N^2(N-1)^2
 \bigl(T_m(W)^2-T_{m-1}(W)T_{m+1}(W)\bigr)\\
 &\quad=((N-1)V-\Sigma)^2
 -\Sigma\bigl(N(N-1)Z-2(N-1)V+\Sigma\bigr)\\
 &\quad=(N-1)\bigl((N-1)V^2-NZ\Sigma\bigr).
\end{align*}
Division by \(N-1>0\) proves the proposition.
\end{proof}

\subsection{Some useful facts}

A finite nonnegative sequence has \emph{interval support} if its positive
terms occur at consecutive indices.

\begin{lemma}\label{lem:superlevel-interval}
Let \(\lambda_0,\ldots,\lambda_d\) be log-concave with interval support.
For every \(C\ge0\), the set
\[
 \{j:\lambda_j>C\}
\]
is an interval of integers, possibly empty.
\end{lemma}

\begin{proof}
On the positive support, the consecutive quotients
\(\lambda_{j+1}/\lambda_j\) are nonincreasing.  Hence the sequence can
change from increasing to decreasing at most once.  Every term between
two terms greater than \(C\) is therefore also greater than \(C\).
\end{proof}

\begin{lemma}\label{lem:crossing}
Let \(d\ge1\), let \(\omega_0,\ldots,\omega_d>0\), and let
\(\lambda_0,\ldots,\lambda_d\ge0\) be log-concave with interval support
\(\{0,1,\ldots,e\}\), where \(1\le e\le d\).  Put
\[
 \pi_j:=\omega_j\lambda_j,
 \qquad
 M(t):=
 \frac{\sum_{j=0}^d j\omega_jt^j}
      {\sum_{j=0}^d\omega_jt^j}
 \quad(t>0),
\]
\[
 \bar\jmath:=
 \frac{\sum_{j=0}^d j\pi_j}
      {\sum_{j=0}^d\pi_j}.
\]
Then:
\begin{enumerate}[label=\textup{(\alph*)}]
\item \(M\) is strictly increasing on \((0,\infty)\), and there is a
      unique \(t>0\) such that \(M(t)=\bar\jmath\).
\item For this \(t\), choose \(\kappa>0\) such that
      \[
       \widetilde\pi_j:=\kappa\omega_jt^j,
       \qquad
       \sum_j\widetilde\pi_j=\sum_j\pi_j,
      \]
      and put \(\Delta_j:=\pi_j-\widetilde\pi_j\).  Then
      \[
       \sum_j\Delta_j=\sum_jj\Delta_j=0,
      \]
      and \(\{j:\Delta_j>0\}\) is an interval.
\item \(\pi_0\le\widetilde\pi_0\).
\item
      \[
       \sum_jj^2\Delta_j\le0,
       \qquad
       \sum_j\frac{\Delta_j}{j+1}\le0.
      \]
\item If \(\rho>0\) and
      \(\lambda_{j+1}\le\rho\lambda_j\) for \(0\le j<d\), then
      \(t\le\rho\).
\end{enumerate}
\end{lemma}

\begin{proof}
For part~\textup{(a)}, direct differentiation gives
\[
 tM'(t)
 =
 \frac{
 \displaystyle
 \sum_{0\le i<j\le d}
 (j-i)^2\omega_i\omega_jt^{i+j}}
 {\left(\sum_{j=0}^d\omega_jt^j\right)^2}
 >0.
\]
Moreover,
\[
 \lim_{t\to0^+}M(t)=0,
 \qquad
 \lim_{t\to\infty}M(t)=d.
\]
Since \(\pi_0,\pi_1>0\), we have \(0<\bar\jmath<d\), proving
part~\textup{(a)}.

For part~\textup{(b)}, take
\[
 \kappa:=
 \frac{\sum_j\pi_j}{\sum_j\omega_jt^j}.
\]
The equality \(\sum_j\Delta_j=0\) follows immediately.  Since
\(M(t)=\bar\jmath\), we also have
\[
 \sum_jj\widetilde\pi_j
 =M(t)\sum_j\widetilde\pi_j
 =\bar\jmath\sum_j\pi_j
 =\sum_jj\pi_j,
\]
and hence \(\sum_jj\Delta_j=0\).  Finally,
\[
 \Delta_j>0
 \quad\Longleftrightarrow\quad
 \lambda_jt^{-j}>\kappa.
\]
The sequence \((\lambda_jt^{-j})\) is log-concave with interval support,
so Lemma~\ref{lem:superlevel-interval} proves the last assertion.

If every \(\Delta_j\) vanishes, parts~\textup{(c)} and \textup{(d)} are
immediate.  Otherwise, the identities in part~\textup{(b)} show that
both signs occur.  Write
\[
 \{j:\Delta_j>0\}=\{r,r+1,\ldots,s\}.
\]
If \(r=0\), then \(\Delta_j(j-s)\le0\) for every \(j\), whereas
\[
 \sum_j\Delta_j(j-s)=0.
\]
This forces \(\Delta_j=0\) for \(j\ne s\), and then
\(\sum_j\Delta_j=0\) also gives \(\Delta_s=0\), a contradiction.  Thus
\(r\ge1\), so \(\Delta_0\le0\), proving part~\textup{(c)}.

For every \(j\),
\[
 \Delta_j(j-r)(j-s)\le0.
\]
Using the two identities in part~\textup{(b)}, we obtain
\[
 \sum_jj^2\Delta_j
 =
 \sum_j\Delta_j(j-r)(j-s)
 \le0.
\]
Also,
\[
 \frac1{j+1}
 =
 \frac{r+s+1-j}{(r+1)(s+1)}
 +
 \frac{(j-r)(j-s)}
 {(r+1)(s+1)(j+1)}.
\]
The first term on the right contributes zero after multiplication by
\(\Delta_j\) and summation.  The second contributes a nonpositive sum,
which proves part~\textup{(d)}.

For part~\textup{(e)}, put
\[
 a_j:=\lambda_j\rho^{-j},
 \qquad
 b_j:=\omega_j\rho^j.
\]
The sequence \((a_j)\) is nonincreasing.  Therefore
\[
 \bar\jmath-M(\rho)
 =
 \frac{
 \displaystyle
 \sum_{0\le i<j\le d}
 (j-i)b_ib_j(a_j-a_i)}
 {\left(\sum_jb_j\right)\left(\sum_jb_ja_j\right)}
 \le0.
\]
Since \(M(t)=\bar\jmath\) and \(M\) is strictly increasing, \(t\le\rho\).
\end{proof}

\subsection{Parity polynomials}

For \(k\ge0\), define
\begin{equation}\label{eq:parity-polynomials}
 \mathcal E_k(x):=\sum_{j\ge0}\binom{k}{2j}x^{2j},
 \qquad
 \mathcal O_k(x):=\sum_{j\ge0}\binom{k}{2j+1}x^{2j+1}.
\end{equation}
Equivalently,
\[
 \mathcal E_k(x)
 =\frac{(1+x)^k+(1-x)^k}{2},
 \qquad
 \mathcal O_k(x)
 =\frac{(1+x)^k-(1-x)^k}{2}.
\]

\begin{lemma}\label{lem:parity-identities}
For \(k\ge2\),
\begin{align}
 \mathcal E_{k-1}(x)^2
 -\mathcal O_{k-2}(x)\mathcal O_k(x)
 &=(1-x^2)^{k-2},\label{eq:parity-first}\\
 \mathcal E_{k-1}(x)\mathcal E_{k+1}(x)
 -\mathcal O_k(x)^2
 &=(1-x^2)^{k-1}.\label{eq:parity-second}
\end{align}
\end{lemma}

\begin{proof}
Put \(X=1+x\) and \(Y=1-x\).  The two left-hand sides, multiplied by
\(4\), are respectively
\[
 (XY)^{k-2}(X+Y)^2
 \quad\text{and}\quad
 (XY)^{k-1}(X+Y)^2.
\]
Since \(XY=1-x^2\) and \(X+Y=2\), the result follows.
\end{proof}

\begin{lemma}\label{lem:reference-sums}
For \(x>0\) and \(L=N-1\),
\begin{align*}
 \sum_{\substack{0\le u\le N\\u\ {\rm odd}}}
 u\binom Nu x^u
 &=Nx\mathcal E_{N-1}(x),\\
 \sum_{\substack{0\le u\le N\\u\ {\rm odd}}}
 u(u-1)\binom Nu x^u
 &=N(N-1)x^2\mathcal O_{N-2}(x),\\
 \sum_{\substack{0\le v\le L\\v\ {\rm odd}}}
 v\binom Lv x^v
 &=Lx\mathcal E_{L-1}(x),\\
 \sum_{\substack{0\le v\le L\\v\ {\rm odd}}}
 \frac1{v+1}\binom Lv x^v
 &=\frac{\mathcal E_{L+1}(x)-1}{(L+1)x}.
\end{align*}
\end{lemma}

\begin{proof}
By \eqref{eq:parity-polynomials},
\[
 x\mathcal O_N'(x)=Nx\mathcal E_{N-1}(x),
 \qquad
 x^2\mathcal O_N''(x)
 =N(N-1)x^2\mathcal O_{N-2}(x).
\]
The third identity follows in the same way.  Finally,
\[
 \sum_{\substack{0\le v\le L\\v\ {\rm odd}}}
 \frac1{v+1}\binom Lv x^v
 =
 \frac1x\int_0^x\mathcal O_L(s)\,ds
 =
 \frac{\mathcal E_{L+1}(x)-1}{(L+1)x}.
\]
\end{proof}

\subsection{Case by case analysis for \(m\) of \(T_m(W)\)}

\subsubsection{\(m=2h\) with \(h\ge1\)}

\begin{proposition}\label{prop:even-centres}
Let \(q,g,W\) be nonnegative integers, and let \(h\ge1\) be an integer.
Then
\[
 T_{2h}(W)^2\ge T_{2h-1}(W)T_{2h+1}(W).
\]
\end{proposition}

\begin{proof}
First suppose that \(q\ge2\) and \(g\ge1\).  Only odd \(u\) occur in
\eqref{eq:w-definition}.  Put
\[
 d:=\left\lfloor\frac{N-1}{2}\right\rfloor,
 \qquad
 \lambda_j:=
 \begin{cases}
  \alpha_{h-j},&0\le j\le h,\\
  0,&j>h.
 \end{cases}
\]
Then
\[
 w_{2j+1}=\binom N{2j+1}\lambda_j
 \qquad(0\le j\le d).
\]
The positive support of \((\lambda_j)_{0\le j\le d}\) is
\(\{0,\ldots,\min(h,d)\}\), whose endpoint is at least \(1\).  For
\(0\le j<\min(h,d)\),
\[
 \frac{\lambda_{j+1}}{\lambda_j}
 =
 \frac{W+h-j}{W+h-j+g-1}
 \le1
\]
and these quotients are nonincreasing.  If \(h<d\), then
\(\lambda_{h+1}=0<\lambda_h\), and all subsequent terms vanish.  Thus
\((\lambda_j)\) is log-concave and
\(\lambda_{j+1}\le\lambda_j\) for \(0\le j<d\).

Apply Lemma~\ref{lem:crossing} with
\[
 \omega_j=\binom N{2j+1}.
\]
Part~\textup{(e)} gives \(0<t\le1\).  Put \(x:=\sqrt t\), and define
\[
 \widetilde w_u:=
 \begin{cases}
  \displaystyle\kappa_{\mathrm e}\binom Nu x^u,&u\ {\rm odd},\\
  0,&u\ {\rm even},
 \end{cases}
 \qquad
 \kappa_{\mathrm e}:=\frac{\kappa}{x}>0,
\]
so that \(\widetilde w_{2j+1}=\widetilde\pi_j\).  Let
\(\widetilde Z,\widetilde V,\widetilde\Sigma\) be defined from
\(\widetilde w\) as in \eqref{eq:moment-sums}.  With
\(\Delta_j=w_{2j+1}-\widetilde w_{2j+1}\), parts~\textup{(b)} and
\textup{(d)} give
\[
\begin{aligned}
 Z-\widetilde Z
 &=\sum_j\Delta_j=0,\\
 V-\widetilde V
 &=\sum_j(2j+1)\Delta_j=0,\\
 \Sigma-\widetilde\Sigma
 &=\sum_j(2j+1)(2j)\Delta_j
 =4\sum_jj^2\Delta_j\le0.
\end{aligned}
\]

Lemma~\ref{lem:reference-sums} and
\eqref{eq:parity-polynomials} give
\[
\begin{aligned}
 \widetilde Z
 &=\kappa_{\mathrm e}\mathcal O_N(x),\\
 \widetilde V
 &=\kappa_{\mathrm e}Nx\mathcal E_{N-1}(x),\\
 \widetilde\Sigma
 &=\kappa_{\mathrm e}N(N-1)x^2\mathcal O_{N-2}(x).
\end{aligned}
\]
Hence, by \eqref{eq:parity-first},
\[
\begin{aligned}
 &(N-1)\widetilde V^2
 -N\widetilde Z\widetilde\Sigma\\
 &\qquad=
 \kappa_{\mathrm e}^2N^2(N-1)x^2
 \bigl(
 \mathcal E_{N-1}(x)^2
 -\mathcal O_N(x)\mathcal O_{N-2}(x)
 \bigr)\\
 &\qquad=
 \kappa_{\mathrm e}^2N^2(N-1)x^2(1-x^2)^{N-2}
 \ge0.
\end{aligned}
\]
Since \(Z=\widetilde Z>0\), \(V=\widetilde V\), and
\(\Sigma\le\widetilde\Sigma\), it follows that
\[
 (N-1)V^2-NZ\Sigma\ge0.
\]
Proposition~\ref{prop:master-identity} now proves the desired inequality.

It remains to treat the other values of \(q\) and \(g\).  If \(g=0\),
then \(T_m(W)=0\) for \(W>0\), while
\(T_m(0)=\binom qm\); the result follows from log-concavity of the
binomial coefficients.  If \(q=0\) and \(g\ge1\), then
\(T_{2h-1}(W)=T_{2h+1}(W)=0\).  Finally, if \(q=1\) and \(g\ge1\), then
\[
 \bigl(T_{2h-1}(W),T_{2h}(W),T_{2h+1}(W)\bigr)
 =
 \bigl(c_{W+h-1},c_{W+h},c_{W+h}\bigr),
\]
and the result follows because \((c_j)\) is nondecreasing.
\end{proof}

\subsubsection{\(m=1\)}

\begin{lemma}\label{lem:m-one}
Assume \eqref{eq:tail-standing-assumptions} and \(g\le\Bc\).  Then, for
every nonnegative integer \(W\),
\[
 T_1(W)^2-T_0(W)T_2(W)
 =
 \alpha_0^2
 \frac{(\Bc-g)+(\Bc-1)W}{W+1}
 \ge0.
\]
\end{lemma}

\begin{proof}
From \eqref{eq:T-definition},
\[
 T_0(W)=\alpha_0,\qquad
 T_1(W)=q\alpha_0,\qquad
 T_2(W)=\binom q2\alpha_0+\alpha_1.
\]
Since \(q^2-\binom q2=\Bc\) and
\(\alpha_1/\alpha_0=(g+W)/(W+1)\), the displayed identity follows.
Its right-hand side is nonnegative because \(g\le\Bc\),
\(\Bc\ge3\), and \(W\ge0\).
\end{proof}

\subsubsection{\(m=2h-1\) with \(h\ge2\)}

\begin{lemma}\label{lem:E-lower-bound}
Let \(L\ge3\), \(0<y<1\), and \(x=\sqrt y\).  Then
\begin{equation}\label{eq:E-lower-bound}
 \mathcal E_{L-1}(x)
 \ge
 (1-y)^{L-3}\bigl(1+(\Bc-2)y\bigr),
 \qquad
 \Bc=\binom L2.
\end{equation}
\end{lemma}

\begin{proof}
Put \(k=L-3\).  The function
\[
 (1-y)^k(1+ky)
\]
is at most \(1\) for \(0\le y<1\); for \(k\ge1\), this follows by
differentiation, and for \(k=0\) it is an equality.  Since
\[
 \Bc-2=\binom{L-1}{2}+L-3,
\]
we obtain
\begin{align*}
 &(1-y)^{L-3}\bigl(1+(\Bc-2)y\bigr)\\
 &\quad=
 (1-y)^k(1+ky)
 +\binom{L-1}{2}y(1-y)^k\\
 &\quad\le
 1+\binom{L-1}{2}y
 \le\mathcal E_{L-1}(x).
\end{align*}
\end{proof}

\begin{proposition}\label{prop:odd-centres}
Assume \eqref{eq:tail-standing-assumptions} and \(g\le\Bc\).  Then
\eqref{eq:universal-tail} holds for every odd \(m\ge3\).
\end{proposition}

\begin{proof}
Write \(m=2h-1\) with \(h\ge2\), and fix a nonnegative integer \(W\).
Only even \(u\)
occur in \eqref{eq:w-definition}, and
\[
 w_{2j}=\binom N{2j}\alpha_{h-j}
 \qquad
 \left(0\le j\le\min\left\{h,\left\lfloor\frac N2\right\rfloor\right\}\right).
\]
For every odd \(v=2j+1\) with \(1\le v\le L\), define
\[
 P_v:=(v+1)w_{v+1}
 =
 \begin{cases}
  \displaystyle
  N\binom Lv\alpha_{h-1-j},&0\le j\le h-1,\\[5pt]
  0,&j>h-1.
 \end{cases}
\]
In the remainder of the proof, every sum over \(v\) is over the odd
integers \(1\le v\le L\).  From \eqref{eq:moment-sums},
\begin{equation}\label{eq:odd-dictionary}
 V=\sum_vP_v,\qquad
 \Sigma=\sum_vvP_v,\qquad
 Z=w_0+\sum_v\frac{P_v}{v+1}.
\end{equation}
Since \(P_1=NL\alpha_{h-1}>0\), we have \(\Sigma>0\).

Put
\[
 A:=W+h-1\ge1.
\]
Then
\begin{equation}\label{eq:theta-definition}
 \theta:=\frac{w_0}{P_1}
 =\frac{\alpha_h}{NL\alpha_{h-1}}
 =\frac{A+g}{NL(A+1)},
 \qquad
 w_0=\theta P_1.
\end{equation}
Since \(N-1=L\), condition \eqref{eq:moment-goal} is equivalent to
\begin{equation}\label{eq:odd-goal}
 \theta P_1+\sum_v\frac{P_v}{v+1}
 \le
 \frac{L\left(\sum_vP_v\right)^2}
 {N\sum_vvP_v}.
\end{equation}

Let
\[
 d:=\left\lfloor\frac{L-1}{2}\right\rfloor,
 \qquad
 \omega_j:=N\binom L{2j+1},
\]
\[
 \lambda_j:=
 \begin{cases}
  \alpha_{h-1-j},&0\le j\le h-1,\\
  0,&j>h-1.
 \end{cases}
\]
Then \(P_{2j+1}=\omega_j\lambda_j\).  Put
\[
 e:=\min(h-1,d).
\]
The positive support of \((\lambda_j)_{0\le j\le d}\) is
\(\{0,\ldots,e\}\), and \(e\ge1\).  For \(0\le j<e\),
\[
 \frac{\lambda_{j+1}}{\lambda_j}
 =
 \frac{A-j}{A-j+g-1}.
\]
These quotients are nonincreasing.  Their largest value is
\[
 y_0:=\frac{A}{A+g-1},
 \qquad
 0<y_0\le1.
\]
If \(e<d\), then \(\lambda_{e+1}=0<\lambda_e\), and all subsequent
terms vanish.  Hence \((\lambda_j)\) is log-concave with interval support,
and
\(\lambda_{j+1}\le y_0\lambda_j\) for \(0\le j<d\).

Apply Lemma~\ref{lem:crossing}.  Part~\textup{(e)} gives
\[
 0<t\le y_0.
\]
Put
\[
 y:=t,\qquad x:=\sqrt t,
\]
and put
\[
 \widetilde P_v:=\kappa_{\mathrm o}\binom Lv x^v
 \qquad(v\ {\rm odd}),
 \qquad
 \kappa_{\mathrm o}:=\frac{\kappa N}{x}>0,
\]
so that \(\widetilde P_{2j+1}=\widetilde\pi_j\).
Parts~\textup{(b)},
\textup{(c)}, and \textup{(d)} give
\begin{equation}\label{eq:odd-comparison-data}
\begin{aligned}
 \sum_v\widetilde P_v
 &=\sum_vP_v,\\
 \sum_vv\widetilde P_v
 &=\sum_vvP_v,\\
 P_1&\le\widetilde P_1,\\
 \sum_v\frac{P_v}{v+1}
 &\le
 \sum_v\frac{\widetilde P_v}{v+1}.
\end{aligned}
\end{equation}
Here the last inequality uses \(v+1=2(j+1)\).
Because \(\theta>0\), it is enough to prove
\[
 \theta\widetilde P_1
 +\sum_v\frac{\widetilde P_v}{v+1}
 \le
 \frac{L\left(\sum_v\widetilde P_v\right)^2}
 {N\sum_vv\widetilde P_v}.
\]

By Lemma~\ref{lem:reference-sums} and
\eqref{eq:parity-polynomials},
\[
\begin{aligned}
 \sum_v\widetilde P_v
 &=\kappa_{\mathrm o}\mathcal O_L(x),\\
 \sum_vv\widetilde P_v
 &=\kappa_{\mathrm o}Lx\mathcal E_{L-1}(x),\\
 \widetilde P_1
 &=\kappa_{\mathrm o}Lx,\\
 \sum_v\frac{\widetilde P_v}{v+1}
 &=\kappa_{\mathrm o}
 \frac{\mathcal E_{L+1}(x)-1}{Nx}.
\end{aligned}
\]
After substitution and cancellation, the required inequality becomes
\[
 \theta LNx^2+\mathcal E_{L+1}(x)-1
 \le
 \frac{\mathcal O_L(x)^2}{\mathcal E_{L-1}(x)}.
\]
By \eqref{eq:parity-second},
\[
 \frac{\mathcal O_L(x)^2}{\mathcal E_{L-1}(x)}
 =
 \mathcal E_{L+1}(x)
 -
 \frac{(1-x^2)^{L-1}}{\mathcal E_{L-1}(x)}.
\]
Thus it remains to prove
\begin{equation}\label{eq:odd-star}
 1-\frac{(1-y)^{L-1}}{\mathcal E_{L-1}(x)}
 \ge\theta LNy.
\end{equation}

If \(y=1\), then \(y\le y_0\le1\) gives \(y_0=1\), hence \(g=1\).
Equation~\eqref{eq:theta-definition} gives \(\theta LN=1\), and
\eqref{eq:odd-star} is an equality.

Assume \(0<y<1\).  Lemma~\ref{lem:E-lower-bound} gives
\[
\begin{aligned}
 1-\frac{(1-y)^{L-1}}{\mathcal E_{L-1}(x)}
 &\ge
 1-\frac{(1-y)^2}{1+(\Bc-2)y}\\
 &=\frac{y(\Bc-y)}{1+(\Bc-2)y}.
\end{aligned}
\]
Set
\[
 F(z):=\frac{\Bc-z}{1+(\Bc-2)z}
 \qquad(0\le z\le1).
\]
Since
\[
 F'(z)=-
 \frac{(\Bc-1)^2}
 {\bigl(1+(\Bc-2)z\bigr)^2}<0,
\]
and \(y\le y_0\), we have \(F(y)\ge F(y_0)\).  A direct calculation
using \(y_0=A/(A+g-1)\) gives
\[
 F(y_0)-\frac{A+g}{A+1}
 =
 \frac{(g-1)(\Bc-g)}
 {(A+1)\bigl((\Bc-1)A+g-1\bigr)}
 \ge0.
\]
Since \(\theta LN=(A+g)/(A+1)\), this proves
\eqref{eq:odd-star}.  It follows that \eqref{eq:odd-goal}, and hence
\eqref{eq:moment-goal}, holds.  Proposition~\ref{prop:master-identity}
now proves \eqref{eq:universal-tail} for every odd \(m\ge3\).
\end{proof}

\subsection{Proof of Theorem~\ref{thm:universal-tail}}

\begin{proof}
The necessity of \(g\le\binom{q+1}{2}\) was obtained from
\(W=0\), \(m=1\), immediately before the statement of
Theorem~\ref{thm:universal-tail}.

For sufficiency, assume \(g\le\binom{q+1}{2}\).  If \(g=0\), then
\[
 T_m(W)=
 \begin{cases}
  \binom qm,&W=0,\\
  0,&W>0,
 \end{cases}
\]
so \eqref{eq:universal-tail} follows from log-concavity of the binomial
coefficients.  This also includes \(q=0\), since the hypothesis forces
\(g=0\) in that case.

If \(q=1\) and \(g\ge1\), then the hypothesis forces \(g=1\).
Consequently \(c_j=1\) and \(T_m(W)=1\) for every \(m,W\ge0\), so the
inequality is an equality.

It remains to assume \(q\ge2\) and \(g\ge1\).  For even \(m\), use
Proposition~\ref{prop:even-centres}; for \(m=1\), use
Lemma~\ref{lem:m-one}; and for odd \(m\ge3\), use
Proposition~\ref{prop:odd-centres}.  This proves sufficiency.
\end{proof}

\subsection{Proof of Theorem~\ref{thm:main-hodge}}

\begin{proof}
Theorem~\ref{thm:universal-tail} and
Proposition~\ref{lem:tail-to-hodge} show that
\textup{(i)} implies \textup{(ii)}.  The implication
\[
 \textup{(ii)}\Longrightarrow\textup{(iii)}
\]
is immediate.

Finally, Lemma~\ref{lem:coefficient-formula} gives
\[
 a_0^{(2)}=1,\qquad
 a_1^{(2)}=q,\qquad
 a_2^{(2)}=\binom q2+g.
\]
Thus the inequality in \textup{(iii)} is equivalent to
\[
 q^2\ge\binom q2+g
 \quad\Longleftrightarrow\quad
 g\le\binom{q+1}{2},
\]
which is \textup{(i)}.
\end{proof}

\section{Proof of Theorem~\ref{thm:unimodality-intro}}
\label{sec:unimodality}

For the implication \textup{(i)}\(\Rightarrow\)\textup{(ii)}, the case
\(g=0\) follows from Remark~\ref{rem:g-zero}.  If
\[
 1\le g\le\binom{q+1}{2},
\]
then \(q\ge1\), every entry of the row is positive by
\eqref{eq:coefficient-formula}, and Theorem~\ref{thm:main-hodge} gives
log-concavity.  Hence the row is unimodal.  It therefore remains, for
this implication, to consider
\[
 g>\binom{q+1}{2}.
\]
Under condition~\textup{(i)}, this gives \(q\ge1\) and \(g\ge2\).

\subsection{The first differences}

Throughout this subsection, assume \(q\ge1\) and \(g\ge2\).  For every
integer \(j\), set
\begin{equation}\label{eq:e-definition}
 e_j:=c_j-c_{j-1}.
\end{equation}

\begin{lemma}\label{lem:e-properties}
We have \(e_j=0\) for \(j<0\), while
\begin{equation}\label{eq:e-generating}
 \sum_{j\ge0}e_jz^j=(1-z)^{1-g}.
\end{equation}
Moreover,
\begin{equation}\label{eq:e-binomial}
 e_j=\binom{g+j-2}{j}>0
 \qquad(j\ge0),
\end{equation}
and
\begin{equation}\label{eq:e-ratio}
 \frac{e_{j-1}}{e_j}
 =\frac{j}{g+j-2}
 \qquad(j\ge1).
\end{equation}
\end{lemma}

\begin{proof}
The assertion for \(j<0\) follows from the convention \(c_j=0\) for
\(j<0\).  By \eqref{eq:c-generating},
\[
 \sum_{j\ge0}e_jz^j
 =(1-z)\sum_{j\ge0}c_jz^j
 =(1-z)^{1-g}.
\]
Equations~\eqref{eq:e-binomial} and \eqref{eq:e-ratio} follow
immediately.
\end{proof}

\begin{proposition}\label{prop:first-difference}
For \(n\ge0\) and \(0\le p\le2n\),
\begin{equation}\label{eq:first-difference}
 a_{p+1}^{(n)}-a_p^{(n)}
 =
 \sum_{j>p-n}
 \binom{q-1}{p+1-2j}e_j
 -
 \binom{q-1}{2n-p}c_{p-n}.
\end{equation}
\end{proposition}

\begin{proof}
Multiplying \eqref{eq:master-series} by \(1-x\) and using
\[
 (1-x)(1+xt)=(1-x^2t)-x(1-t),
\]
we obtain
\begin{align*}
 &(1-x)\sum_{n\ge0}\sum_{p\ge0}a_p^{(n)}x^pt^n\\
 &\qquad=
 \frac{(1+xt)^{q-1}}{(1-t)(1-x^2t)^{g-1}}
 -
 \frac{x(1+xt)^{q-1}}{(1-x^2t)^g}.
\end{align*}
The coefficient of \(x^{p+1}t^n\) on the left is
\(a_{p+1}^{(n)}-a_p^{(n)}\).

By \eqref{eq:e-generating}, the first term on the right is
\[
 \left(\sum_{i=0}^{q-1}\binom{q-1}{i}x^it^i\right)
 \left(\sum_{j\ge0}e_jx^{2j}t^j\right)
 \left(\sum_{k\ge0}t^k\right).
\]
A term contributes to the coefficient of \(x^{p+1}t^n\) precisely when
\[
 i+2j=p+1,
 \qquad
 i+j+k=n.
\]
Thus
\[
 i=p+1-2j,
 \qquad
 k=n-p-1+j,
\]
and \(k\ge0\) is equivalent to \(j>p-n\).  Since \(e_j=0\) for
\(j<0\), the contribution is
\[
 \sum_{j>p-n}\binom{q-1}{p+1-2j}e_j.
\]

For the second term on the right, the corresponding equations are
\[
 i+2j=p,
 \qquad
 i+j=n.
\]
Their unique solution is
\[
 i=2n-p,
 \qquad
 j=p-n,
\]
so its contribution is
\[
 \binom{q-1}{2n-p}c_{p-n}.
\]
Subtracting the two contributions proves \eqref{eq:first-difference}.
\end{proof}

If \(n\ge1\) and \(0\le p\le n-1\), then \(c_{p-n}=0\), and every
\(j\ge0\) satisfies \(j>p-n\).  Hence
\begin{equation}\label{eq:left-half-increasing}
 a_{p+1}^{(n)}-a_p^{(n)}
 =
 \sum_{j\ge0}\binom{q-1}{p+1-2j}e_j
 \ge0.
\end{equation}
Consequently,
\[
 a_0^{(n)}\le a_1^{(n)}\le\cdots\le a_n^{(n)}.
\]

It remains to consider \(n\le p\le2n-1\).  Put
\[
 m:=2n-p,
\]
so that \(1\le m\le n\).  In \eqref{eq:first-difference}, set
\(j=n-m+t\).  Then \(j>p-n=n-m\) is equivalent to \(t\ge1\), and
\[
 p+1-2j=m+1-2t.
\]
For \(1\le m\le n\), define
\begin{equation}\label{eq:ND-definition}
 \mathcal N_m
 :=
 \sum_{t\ge1}
 \binom{q-1}{m+1-2t}e_{n-m+t},
 \qquad
 \mathcal D_m
 :=
 \binom{q-1}{m}c_{n-m}.
\end{equation}
Then
\begin{equation}\label{eq:tail-difference}
 a_{2n-m+1}^{(n)}-a_{2n-m}^{(n)}
 =\mathcal N_m-\mathcal D_m
 \qquad(1\le m\le n).
\end{equation}
Moreover, \(\mathcal D_m>0\) for
\(1\le m\le\min\{n,q-1\}\), whereas \(\mathcal D_m=0\) for
\(q\le m\le n\).

\begin{lemma}\label{lem:ND-ratio}
Assume \(n\ge 1,\;g\ge2\) and
\[
 1\le m<\min\{n,q-1\}.
\]
Then
\begin{equation}\label{eq:ND-ratio}
 \mathcal N_{m+1}\mathcal D_m
 \ge
 \mathcal N_m\mathcal D_{m+1}.
\end{equation}
\end{lemma}

\begin{proof}
In this range, \(\mathcal D_m\) and \(\mathcal D_{m+1}\) are positive,
and \eqref{eq:c-ratio} gives
\[
 \frac{\mathcal D_{m+1}}{\mathcal D_m}
 =
 \frac{q-1-m}{m+1}
 \frac{n-m}{g+n-m-1}.
\]

Consider a nonzero term of \(\mathcal N_m\) corresponding to \(t\ge1\).
The term of \(\mathcal N_{m+1}\) having the same value of \(t\) is also
nonzero.  Indeed,
\[
 0\le m+1-2t\le m-1,
 \qquad
 n-m+t\ge2,
\]
and hence
\[
 1\le m+2-2t\le m\le q-2,
 \qquad
 n-m-1+t\ge1.
\]
Therefore \eqref{eq:e-ratio} gives
\begin{align*}
 &\frac{\binom{q-1}{m+2-2t}e_{n-m-1+t}}
 {\binom{q-1}{m+1-2t}e_{n-m+t}}\\
 &\qquad=
 \frac{q-m-2+2t}{m+2-2t}
 \frac{n-m+t}{g+n-m+t-2}.
\end{align*}
Since \(t\ge1\),
\begin{align*}
 &\frac{q-m-2+2t}{m+2-2t}-\frac{q-m}{m}
 =
 \frac{2q(t-1)}{m(m+2-2t)}
 \ge0,\\
 &\frac{n-m+t}{g+n-m+t-2}
 -\frac{n-m+1}{g+n-m-1}
 =
 \frac{(g-2)(t-1)}
 {(g+n-m+t-2)(g+n-m-1)}
 \ge0.
\end{align*}
Therefore
\begin{align*}
 &\frac{\binom{q-1}{m+2-2t}e_{n-m-1+t}}
 {\binom{q-1}{m+1-2t}e_{n-m+t}}\\
 &\qquad\ge
 \frac{q-m}{m}
 \frac{n-m+1}{g+n-m-1}\\
 &\qquad\ge
 \frac{q-1-m}{m+1}
 \frac{n-m}{g+n-m-1}\\
 &\qquad=
 \frac{\mathcal D_{m+1}}{\mathcal D_m}.
\end{align*}
Here the second inequality follows from
\[
 \frac{q-m}{m}-\frac{q-1-m}{m+1}
 =\frac{q}{m(m+1)}>0
\]
and \(n-m+1\ge n-m\).

Thus every nonzero term of \(\mathcal N_m\) is matched with a term of
\(\mathcal N_{m+1}\) which is at least
\(\mathcal D_{m+1}/\mathcal D_m\) times as large.  Summing these
inequalities and adding any remaining nonnegative terms of
\(\mathcal N_{m+1}\), we obtain
\[
 \mathcal N_{m+1}
 \ge
 \frac{\mathcal D_{m+1}}{\mathcal D_m}\mathcal N_m.
\]
This proves \eqref{eq:ND-ratio}.
\end{proof}

\subsection{Proof of Theorem~\ref{thm:unimodality-intro}}

\begin{proof}
\textup{(i)}\(\Rightarrow\)\textup{(ii)}.
The cases \(g=0\) and
\(1\le g\le\binom{q+1}{2}\) were proved above.  Suppose now that
\[
 g>\binom{q+1}{2}.
\]
Condition~\textup{(i)} then gives \(q\ge1\), and consequently \(g\ge2\).

When \(n=0\), the row consists only of \(a_0^{(0)}\) and hence is
unimodal.  Let \(n\ge1\).
Equation~\eqref{eq:left-half-increasing} shows that
\[
 a_0^{(n)}\le a_1^{(n)}\le\cdots\le a_n^{(n)}.
\]

We claim that the integers \(m\in\{1,\ldots,n\}\) for which
\[
 \mathcal N_m\ge\mathcal D_m
\]
form a \emph{final segment} of \(\{1,\ldots,n\}\): if $m\in\{1,\ldots,n\}$ such that $\mathcal N_m\ge\mathcal D_m$ and
$m<m'\le n$, then $\mathcal N_{m'}\ge\mathcal D_{m'}$. 

Suppose that \(1\le m<n\)
and \(\mathcal N_m\ge\mathcal D_m\).  If \(m+1\ge q\), then
\(\mathcal D_{m+1}=0\), and hence
\(\mathcal N_{m+1}\ge\mathcal D_{m+1}\).  If \(m+1\le q-1\), then
Lemma~\ref{lem:ND-ratio} applies and gives
\[
 \mathcal N_{m+1}\mathcal D_m
 \ge
 \mathcal N_m\mathcal D_{m+1}
 \ge
 \mathcal D_m\mathcal D_{m+1}.
\]
Since \(\mathcal D_m>0\), it follows again that
\(\mathcal N_{m+1}\ge\mathcal D_{m+1}\).  This proves the claim.

By \eqref{eq:tail-difference}, the signs of the differences with
\(n\le p\le2n-1\) are determined by
\(\mathcal N_m-\mathcal D_m\), where \(m=2n-p\).  Since \(m\) decreases
as \(p\) increases, these differences are first nonnegative and then
negative, with at most one change; either part may be empty.  Together
with \eqref{eq:left-half-increasing}, this proves that
\[
 \left(a_p^{(n)}\right)_{p=0}^{2n}
 =
 \left(h^{p,0}(S^{[n]})\right)_{p=0}^{2n}
\]
is unimodal.  Thus \textup{(i)} implies \textup{(ii)}.

The implication \textup{(ii)}\(\Rightarrow\)\textup{(iii)} follows by
taking \(n=1\).

Finally, assume \textup{(iii)}.  Since \(S^{[1]}\cong S\), the row is
\[
 \left(h^{0,0}(S),h^{1,0}(S),h^{2,0}(S)\right)
 =(1,q,g).
\]
If \textup{(i)} does not hold, then \(q=0\) and \(g>0\), so this row is
\((1,0,g)\), which is not unimodal.  Therefore \textup{(iii)} implies
\textup{(i)}.
\end{proof}

\begin{remark}[Why \(q\ge1\) is necessary]\label{rem:q-zero-unimodality}
Suppose that \(q=0\) and \(g\ge1\).  In
\eqref{eq:coefficient-formula}, the binomial coefficient can be nonzero
only when \(i=0\).  Hence
\[
 a_{2j}^{(n)}=c_j>0
 \qquad(0\le j\le n),
\]
whereas
\[
 a_{2j+1}^{(n)}=0
 \qquad(0\le j<n).
\]
Thus, for \(n\ge1\), the row has a zero between positive terms and is
not unimodal.
\end{remark}

\section*{Acknowledgements}

The author wishes to thank the Department of Mathematics, IIT Madras, for excellent working conditions, and the Prime Minister's Research Fellowship (PMRF, ID: 2503482) for their financial support.

\end{document}